\documentclass[11pt,a4paper]{amsart}

\usepackage{fontenc}[T1]
\usepackage[singlespacing]{setspace}
\usepackage{hyperref}
\usepackage{mathtools}
\usepackage{amsthm}
\usepackage{amssymb}
\usepackage{cleveref}
\usepackage{enumitem}
\usepackage{xcolor}
\usepackage{todonotes}
\usepackage{tikz, ifthen}

\theoremstyle{plain}
\newtheorem{thm}{Theorem}[section]
\newtheorem{lem}[thm]{Lemma}
\newtheorem{prop}[thm]{Proposition}
\newtheorem{cor}[thm]{Corollary}

\theoremstyle{definition}
\newtheorem{defn}[thm]{Definition}
\newtheorem{eg}[thm]{Example}

\theoremstyle{remark}
\newtheorem{rmk}[thm]{Remark}

\DeclareMathOperator{\Stab}{Stab}
\DeclareMathOperator{\Endo}{End}
\DeclareMathOperator{\Auto}{Aut}
\DeclareMathOperator{\EllO}{Ell_{\mathcal{O}}}

\title{Ordinary Isogeny Graphs with Level Structure}

\author{Derek Perrin}
\author{Jos\'e Felipe Voloch}
\address{School of Mathematics and Statistics, University of Canterbury, Private Bag 4800, Christchurch 8140, New Zealand}
\email{derek.perrin@pg.canterbury.ac.nz}
\email{felipe.voloch@canterbury.ac.nz}
\thanks{The authors thank W. Castryck, S. Galbraith, and the anonymous referee for comments and MBIE for support}

\begin{document}

\maketitle

\begin{abstract}
  We study $\ell$-isogeny graphs of ordinary elliptic curves defined over $\mathbb{F}_q$ with an added level structure. Given an integer $N$ coprime to $p$ and $\ell,$ we look at the graphs obtained by adding $\Gamma_0(N),$ $\Gamma_1(N),$ and $\Gamma(N)$-level structures to volcanoes. Given an order $\mathcal{O}$ in an imaginary quadratic field $K,$ we look at the action of generalised ideal class groups of $\mathcal{O}$ on the set of elliptic curves whose endomorphism rings are $\mathcal{O}$ along with a given level structure. We show how the structure of the craters of these graphs is determined by the choice of parameters.
\end{abstract}

\section{Introduction}\label{sec:introduction}
 The set of isomorphism classes of elliptic curves over a finite field $\mathbb{F}_q$ can be viewed as vertices of a graph known as the isogeny graph. Vertices $E$ and $E'$ are connected by a directed edge $(E, E')$ if and only if there exists an isogeny $\varphi: E \to E'.$ Advancements in quantum computing have motivated the search for new cryptographic primitives which are resistant to quantum-style attacks. The search for problems which are computationally difficult for quantum adversaries gave birth to what is known as \emph{isogeny-based cryptography}~\cite{charles-lauter-goren}. Given two vertices $E$ and $E',$ the underlying hard problem of isogeny-based cryptography is that of finding a path from $E$ to $E'$ in the isogeny graph.

Cryptographic applications of isogenies has attracted more attention to studying them. One such development in the field of isogenies has been to look at adding a level $N$ structure to the isogeny graph. This graph consists of vertices of the form $(E, \gamma)$ where $\gamma$ is some level $N$ structure on $E,$ and there is an edge between two vertices $(E, \gamma),$ $(E', \gamma')$ if and only if there exists an isogeny $\varphi: E \to E'$ with the added condition that $\varphi(\gamma) = \gamma'.$ Isogeny graphs with level $N$ structure were first studied by Roda in~\cite{Roda-thesis} in which they looked at adding a full level $N$ structure to the supersingular component of the $\ell$-isogeny graph. Following Roda, further works have also studied adding level structure to the supersingular component; see~\cite{Shai,Arpin-level-structure}. Additionally, during the writing of this paper, Arpin, Castryck, Eriksen, Lorenzon, and Vercauteren~\cite{class-group-action-oriented} published results involving generalised class group actions on oriented elliptic curves with level structure. Prior to that, Col\`o~\cite{Colo} pointed out that the action of the ideals on elliptic curves with level structure factors through a ray class group.

Parallel to these developments, the isogeny graphs for ordinary elliptic curves was studied, motivated by problems in computational number theory. Kohel~\cite{Kohel} showed that the ordinary components have a nice ``volcano'' structure, where the term volcano was later coined by Fouquet and Morain~\cite{Fouquet-Morain} to describe the components due to their similarities to geological volcanoes. Sutherland~\cite{Sutherland-volcanoes} provides a summary of various algorithms, including computing the endomorphism ring of ordinary elliptic curves, which make use of volcanoes.

Due to their applications in cryptography, the supersingular setting has received more attention and the ordinary setting is not as well-studied. However, recent work by Lei and M\"uller~\cite{Lei-Muller} studies ordinary isogeny graphs with a $\Gamma_1(Np^m)$-level structure. Our work differs as we will consider ordinary isogeny graphs with $\Gamma_0(N),$ $\Gamma_1(N),$ and $\Gamma(N)$-level structures.

The viewpoint of our work is more mathematical in nature, and we will not discuss cryptographic applications here. For use of level structures in cryptography, the interested reader is recommended to look at~\cite{CSIDH-w-LS, defeo-level-structure}. This paper aims to answer a question first raised by Levin~\cite{Shai}, namely that of determining the crater size and number of components of the graphs constructed by adding various level structures to an ordinary component of the $\ell$-isogeny graph.

The theory of complex multiplication tells us that for an order $\mathcal{O}$ in an imaginary quadratic field $K,$ the ideal class group $Cl(\mathcal{O})$ acts simply transitively on the set
\[
  \EllO(\mathbb{F}_{q}) = \left\{ E/\mathbb{F}_q \mid \Endo(E) \simeq \mathcal{O} \right\} /\,\overline{\mathbb{F}_q}\text{-isomorphism}
\]
of ordinary elliptic curves whose endomorphism rings are isomorphic to $\mathcal{O}$ modulo $\overline{\mathbb{F}_q}$-isomorphism. The main contribution of this work can be thought of as a generalisation of this group action. We address Levin's question by considering generalised ideal class groups acting on $\EllO(\mathbb{F}_q)$ endowed with various level structures. For $\mathfrak{l}$ a prime $\mathcal{O}$-ideal above $\ell,$ we show that the crater size of components with a level structure is related to the order of the subgroup $\left\langle \mathfrak{l}, \overline{\mathfrak{l}} \right\rangle$ when viewed as a subgroup of a generalised class group. In \Cref{sec:class-field-theory}, we give descriptions of these class groups for $\Gamma_0(N),$ $\Gamma_1(N),$ and $\Gamma(N)$-level structures. We also show the role the prime factorisation of $N$ as an $\mathcal{O}$-ideal plays in both the number and size of the components of these graphs with level structure. In \Cref{sec:examples}, we conclude by showing several examples of adding $\Gamma_0(N)$ and $\Gamma_1(N)$-level structures to an isogeny volcano.

\section{Background}\label{sec:background}
\begin{defn}
  Let $\mathbb{F}_q$ be a finite field of characteristic $p \neq 2,3$ and $\ell$ a prime different from $p.$ Recall that the set of $\overline{\mathbb{F}_q}$-isomorphism classes of elliptic curves over $\mathbb{F}_q$ can be identified with $\mathbb{F}_q$ via the $j$-invariant. We let $X_{\ell}(\mathbb{F}_q)$ denote the \emph{$\ell$-isogeny graph of $\mathbb{F}_q$} where $X_{\ell}(\mathbb{F}_q)$ has vertex set $\mathbb{F}_q.$ Each vertex represents an isomorphism class of elliptic curves defined over $\mathbb{F}_q.$ The edge set of $X_{\ell}(\mathbb{F}_q)$ consists of pairs $(j_1, j_2)$ where $j_1, j_2$ are $j$-invariants of elliptic curves which are $\overline{\mathbb{F}_q}$-isogeneous via an isogeny of degree $\ell.$
\end{defn}
The field we work with will generally be understood to be $\mathbb{F}_q,$ so we will write $X_{\ell}$ instead of $X_{\ell}(\mathbb{F}_q).$ Unless otherwise stated, we will treat $X_{\ell}$ as an undirected graph since given an isogeny $\varphi: E \to E'$ of degree $\ell,$ there exists a dual isogeny $\hat{\varphi}: E' \to E$ of degree $\ell$ such that $\varphi \circ \hat{\varphi} = \hat{\varphi} \circ \varphi = [\ell].$

For the purposes of this paper, we will focus our attention on the ordinary components, volcanoes, of $X_{\ell}.$
\begin{defn}
  An \emph{$\ell$-volcano} is a connected, undirected graph whose vertices may be partitioned into one or more levels $V_0, \cdots, V_d$  such that:
  \begin{enumerate}[label=(\roman*)]
  \item The subgraph on $V_0$ is a regular connected graph of degree at most 2.
  \item For each $i > 0,$ the vertices in $V_i$ have exactly one neighbour in $V_{i-1}$ and the subgraph on $V_i$ is totally disconnected.
  \item For each $i < d,$ the vertices in $V_i$ have degree $\ell + 1.$
  \end{enumerate}
\end{defn}
\begin{defn}
Let $G_{\ell}$ be an $\ell$-isogeny volcano of $X_{\ell}.$ The subgraph of level zero on $G_{\ell}$ is called the \emph{crater of $G_{\ell}$} and is denoted $C_\ell$ and the subgraph of $X_{\ell}$ consisting of the craters of all ordinary components will be denoted $\mathcal{C}_{\ell}$ and is called the \emph{crater of $X_{\ell}$}.
\end{defn}
The term \emph{crater} was introduced in~\cite{Fouquet-Morain} and strictly refers to the level zero subgraph of a single volcano. Some definitions of a volcano \cite{distort-volcano} require the crater to consist of a cycle, but we do not require this in our work and use a definition as in \cite{Sutherland-volcanoes}. Many of the graphs we work with contain multiple components with a volcanic structure, so it makes sense to discuss the craters of all of $X_{\ell}.$ This terminology is consistent with what the authors of~\cite{Lei-Muller} use.

We can now add an additional structure to the ordinary components of $X_{\ell}.$ The structure we impose is called a \emph{level structure} and it is related to the $N$-torsion subgroup $E[N]$ of an elliptic curve $E.$
\begin{defn}\label{def:level-structures}
  Let $E/\mathbb{F}_q$ be an elliptic curve, $N$ an integer, and $E[N]$ the group of $N$-torsion points of $E$. Then we call
  \begin{enumerate}
    \item a \emph{$\Gamma_0(N)$-level structure} a cyclic subgroup $\left\langle P \right\rangle \subseteq E[N]$ of order $N$;
    \item a \emph{$\Gamma_1(N)$-level structure} a point $P \in E[N]$ of order $N;$
    \item and a \emph{$\Gamma(N)$-level structure}, or \emph{full level structure}, an ordered pair of points $P, Q \in E[N]$ which are a basis for $E[N].$
  \end{enumerate}
\end{defn}
In this paper, we will always assume $N$ is an integer coprime to $p$ and $\ell.$
\begin{defn}\label{def:vertex-equivalence}
  Consider the set of pairs $(E, \gamma)$ where $E$ is an ordinary elliptic curve and $\gamma$ is a level structure on $E.$ Then we may define an equivalence relation $\sim$ on pairs $(E, \gamma),$  where $(E, \gamma) \sim (E', \gamma')$ if and only if there exists an isomorphism $\varphi: E \to E'$ such that $\varphi(\gamma) = \gamma'.$
\end{defn}
Note that if $\varphi \in \Auto(E),$ then $(E, \gamma)$ is equivalent to $(E, \varphi(\gamma)).$ In particular this applies to $[-1].$

We add a level structure to $X_{\ell}$ by adding the level structure to each vertex in $X_{\ell}$ up to the equivalence relation $\sim$ in \Cref{def:vertex-equivalence}.
\begin{defn}\label{defn:level-struct-volcanoes}
  Let $G_{\ell}$ be an isogeny volcano with vertex set $\mathbb{V} = {E_i}$. Up to the equivalence relation in \cref{def:vertex-equivalence}, consider the set of tuples
  \[
    \mathbb{V}' = \left\{(E_i, \gamma_i) \mid \forall E_i \in \mathbb{V} \text{ where } \gamma_i \text{ is a level structure on } E_i \right\} / \sim
  \]
  and the set of tuples
  \begin{align*}
    \mathbb{E}' = \bigl\{&((E_i, \gamma_i), (E_j, \gamma_j)) \text{ if } \exists \varphi: E_i \to E_j \text{ such that } \deg\varphi = \ell\\
    & \text{ and } \varphi(\gamma_i) = \varphi(\gamma_j) \bigr\} \big/ \sim.
  \end{align*}
  Then
  \begin{enumerate}[label=(\roman*)]
  \item if the $\gamma_i$ are all $\Gamma_0(N)$ level structures on the $E_i$'s, we let $G_{\ell,0}(N)$ denote the graph whose vertex set is $\mathbb{V}'$ and edge set is $\mathbb{E}'$;
  \item if the $\gamma_i$ are all $\Gamma_1(N)$ level structures on the $E_i$'s, we let $G_{\ell,1}(N)$ denote the graph whose vertex set is $\mathbb{V}'$ and edge set is $\mathbb{E}'$;

  \item if the $\gamma_i$ are all $\Gamma(N)$ level structures on the $E_i$'s, we let $G_{\ell}(N)$ denote the graph whose vertex set is $\mathbb{V}'$ and edge set is $\mathbb{E}'$.
  \end{enumerate}
\end{defn}
We will say we \emph{add a $\gamma$ level structure to $G_{\ell}$} when we construct one of the graphs in \Cref{defn:level-struct-volcanoes}. We also remark that while $G_{\ell,0}(N)$ remains undirected, the  graphs $G_{\ell,1}(N)$ and $G_{\ell}(N)$ in \Cref{defn:level-struct-volcanoes} are no longer undirected graphs.
\section{Stabilisers of Level Structures}\label{sec:stabilisers}
In this section, we will discuss the stabilisers of the $\Gamma_0(N),$ $\Gamma_1(N),$ and $\Gamma(N)$-level structures.

We will let $K$ be an imaginary quadratic field with ring of integers given by $\mathcal{O}_K.$ We are interested in ordinary elliptic curves with complex multiplication (CM) by an order $\mathcal{O} \subseteq \mathcal{O}_K$ where $\mathcal{O} = \mathbb{Z}[\Phi]$ for some $\Phi \in \mathcal{O}_{K};$ in other words, their endomorphism ring is equal to $\mathcal{O}.$ We will restrict ourselves to the case of $(f, N) = 1$ where $f$ is the conductor of $\mathcal{O}$ (i.e. $f = [\mathcal{O}_K : \mathcal{O}]$). Finally, we will assume that all elliptic curves we work with have $j$-invariant not equal to $0, 1728.$
\begin{defn}
Let $\gamma$ be a $\Gamma_0(N),$ $\Gamma_1(N),$ or $\Gamma(N)$-level structure on an ordinary elliptic curve $E.$ A \emph{stabiliser of $\gamma$} is an element $\alpha$ of $\mathcal{O}/N \mathcal{O}$ where $(E, \gamma) \sim (E, \alpha(\gamma))$ where $\sim$ is the equivalence relation defined in \Cref{def:vertex-equivalence}. The set $\Stab(\gamma)$ of stabilisers of $\gamma$ form a subgroup of $(\mathcal{O}/N \mathcal{O})^{\times}.$
\end{defn}
\begin{lem}\label{lem:O/N isomorphism}
  The $N$-torsion subgroup $E[N]$ and $\mathcal{O}/N\mathcal{O}$ are isomorphic as $\mathcal{O}/N\mathcal{O}$-modules.
\end{lem}
\begin{proof}
 A proof when $\mathcal{O} = \mathcal{O}_K$ is given in~\cite[Prop.~II.1.4]{AAEC}. For $\mathcal{O}$ a non-maximal order, a proof is given in~\cite[3.1~Lemma~1]{class-group-action-oriented}.
\end{proof}

The importance of \Cref{lem:O/N isomorphism} is that it allows us to work with a somewhat simpler structure as we will shortly see. For the three level structures $\Gamma(N),$ $\Gamma_1(N),$ and $\Gamma_0(N),$ we are interested in stabilisers of all of $E[N],$ points $P \in E[N]$ of order $N,$ and cyclic subgroups $\left\langle P \right\rangle \subseteq E[N]$ of order $N$ respectively. The authors of~\cite{CSIDH-w-LS} find the stabilisers of a $\Gamma(N)$-level structure by considering subgroups of the kernel of an isogeny which fixes all of $E[N];$ this method can be adapted for $\Gamma_1(N)$ and $\Gamma_0(N)$-level structures, but one needs to take into consideration eigenvectors of $\Phi.$ We instead look at the equivalent problem of finding stabilisers for all of $\mathcal{O}/N\mathcal{O},$ elements $\alpha \in \mathcal{O}/N\mathcal{O}$ of order $N,$ and cyclic subgroups $\left\langle \alpha \right\rangle \subseteq \mathcal{O}/N\mathcal{O}$ of order $N$ respectively.

\begin{prop}\label{prop:stab-full-level}
  The stabilisers of a $\Gamma(N)$-level structure are given by
  \[
    \Stab(\mathcal{O}/N\mathcal{O}) = \left\{ \alpha \in (\mathcal{O}/N \mathcal{O})^{\times} \mid \alpha \equiv \pm 1 \mod N\mathcal{O} \right\}.
  \]
\end{prop}
\begin{proof}
  Due to the equivalence relation in \Cref{def:vertex-equivalence}, we trivially have $\alpha \equiv \pm1 \mod N\mathcal{O}.$
\end{proof}
Let $\tilde{P} \in E[N]$ be a $\Gamma_1(N)$-level structure. After choosing a generator of $E[N]$ as an $\mathcal{O}/N \mathcal{O}$-module, then by \Cref{lem:O/N isomorphism} we may identify $\tilde{P}$ with an element $P \in \mathcal{O}/N\mathcal{O}.$ In what follows, we will work with $P$ instead of $\tilde{P}.$
\begin{prop}\label{prop:stab-gamma-1}
  Given $P \in \mathcal{O}/N\mathcal{O}$ of order $N$ and the ideal $\mathfrak{a} = \langle P, N\mathcal{O} \rangle,$ then
    \[
      \Stab(P) = \left\{ \alpha \in (\mathcal{O}/N \mathcal{O})^{\times} \mid \alpha \equiv \pm 1 \mod \mathfrak{a}^{-1}N\mathcal{O} \right\}.
    \]
\end{prop}
\begin{proof}
  Let $\alpha \in \Stab(P).$ Then
  \begin{align*}
    \alpha P &\equiv \pm P \mod N\mathcal{O} \\
    (\alpha \pm 1)P &\equiv 0 \mod N\mathcal{O}
  \end{align*}
  Then $(\alpha \pm 1)P \in N \mathcal{O}.$ For any element $\beta = aP + bN \in \mathfrak{a}$ for $a,b \in \mathcal{O},$ consider the product
  \begin{align*}
    (\alpha \pm 1)\beta &= (\alpha \pm 1)(aP + bN) \\
                        &= a(\alpha \pm 1)P + b(\alpha \pm 1)N.
  \end{align*}
  We have shown the first term is in $N \mathcal{O},$ and the second term is certainly in $N \mathcal{O},$ so their sum is in $N \mathcal{O}.$ Then $(\alpha \pm 1)\mathfrak{a} \subseteq N \mathcal{O},$ and so $(\alpha \pm 1) \in \mathfrak{a}^{-1}N \mathcal{O},$ and $\Stab(P) \subseteq \left\{ \alpha \in \mathcal{O}/N \mathcal{O} \mid \alpha \equiv \pm 1 \mod \mathfrak{a}^{-1}N \mathcal{O} \right\}.$ The reverse containment is clear.
  \end{proof}

\begin{prop}\label{prop:stab-gamma-0}
  Let $G = \left\langle P \right\rangle$ be a $\Gamma_0(N)$-level structure and $\mathfrak{a} = \left\langle P, N \mathcal{O} \right\rangle$ the smallest ideal containing $P$ and $N.$ Then
    \[
      \Stab(G) = \left\{ \alpha \in (\mathcal{O}/N \mathcal{O})^{\times} \mid \alpha \equiv c \mod \mathfrak{a}^{-1}(N) \text{ for } c \in \mathbb{Z},\, (c, N) = 1\right\}.
    \]
\end{prop}
\begin{proof}
  Let $\alpha G = G.$ Then $\alpha P = P'$ where $G = \left\langle P' \right\rangle.$ Since $G$ is cyclic of order $N$, we must have $P' = c P$ for some $c$ coprime to $N.$ The rest of the proof follows as \Cref{prop:stab-gamma-1} but by replacing the requirement of $\alpha \equiv \pm 1 \mod \mathfrak{a}^{-1}N \mathcal{O}$ with $\alpha \equiv c \mod \mathfrak{a}^{-1}N \mathcal{O}$ where $c$ is an integer coprime to $N.$
\end{proof}
\begin{rmk}
  The above propositions can be modified to work for $j$-invariants $0$ and $1728$ by considering congruences modulo a $6^{\text{th}}$ or $4^{\text{th}}$ root of unity respectively.
\end{rmk}
\section{Adding Level Structure to Isogeny Graphs}\label{sec:level-structures}
This section will be organised into three subsections where we will look at ordinary $\ell$-isogeny graphs with the three types of level structures. We are particularly interested in how the choice of $N$ affects the size and number of components when adding level structure.  In what follows, we will let $\mathcal{O}$ be an order in an imaginary quadratic field $K.$ Rather than look at the entire isogeny graph, we will focus our attention on the connected components (i.e. volcanoes), and in particular, their craters.  We will let $G_{\ell}$ be such a component whose crater $C_{\ell}$ consists of isomorphism classes of elliptic curves whose endomorphism rings are equal to $\mathcal{O}.$ Further, we will restrict ourselves to only working with elliptic curves whose $j$-invariants are not $0$ or $1728$ to avoid automorphism groups different from $\left\{ \pm 1 \right\}.$ As usual, we will let $N$ be a positive integer coprime to $p,$ $\ell,$ and the conductor of $\mathcal{O}.$

\begin{lem}\label{lem:volcano-growth}
  Let $G_{\ell}$ be an isogeny volcano with $n$ vertices.
  \begin{enumerate}[label=(\roman*)]
  \item The graph $G_{\ell,0}(N)$ has
    \[
      n N\prod_{p \mid N}\left( 1 + \frac{1}{p} \right)
    \]
    vertices.
  \item If $N > 2,$ the graph $G_{\ell,1}(N)$ has
    \[
      \frac{nN\phi(N)}{2}\prod_{p \mid N}\left( 1 + \frac{1}{p} \right)
    \]
    vertices and $3n$ if $N = 2$.
  \item If $N > 2,$ the graph $G_{\ell}(N)$ has
    \[
      \frac{nN^2\phi(N)^2}{2} \prod_{p \mid N}\left( 1 + \frac{1}{p} \right)
    \]
    vertices and $6n$ if $N = 2$.
  \end{enumerate}
  where $\phi$ is Euler's totient function.
\end{lem}
\begin{proof}
  For each case, we proceed by counting the number of level structures which can be added to each vertex in $G_{\ell}.$

  We begin by proving (ii). Since $E[N] \simeq \mathbb{Z}/N\mathbb{Z} \times \mathbb{Z}/N \mathbb{Z}$ and $N = \prod p_i^{e_i},$ then by the Chinese Remainder Theorem we may write
  \begin{equation}\label{eq:torsion-factors}
    E[N] \simeq \left( \mathbb{Z}/p_1^{e_1}\mathbb{Z} \right)^2 \times \left( \mathbb{Z}/p_2^{e_2}\mathbb{Z} \right)^2 \times \cdots \times \left( \mathbb{Z}/p_r^{e_r}\mathbb{Z} \right)^2.
  \end{equation}
   An element $P$ of $E[N]$ has maximal order $N$ if and only if it has maximal order in each of factors on the right-hand side of \eqref{eq:torsion-factors}, and so it is sufficient to find the number of elements of maximal order of each factor of $E[N].$ An element $(a,b)$ of $(\mathbb{Z}/p_i^{e_i})^2$ is of maximal order if either $(a,p_i^{e_i}) = 1$ or $(b, p_i^{e_i}) = 1.$ There are $\phi(p_i^{e_i})p_i^{e_i}$ pairs where $a$ is coprime to $p_i^{e_i}$ and $\phi(p_i^{e_i})(p_i^{e_i} - \phi(p_i^{e_i}))$ pairs where $b$ is coprime to $p_i^{e_i}$ and $a$ is not. Some algebraic manipulation shows there are a total of $p_i^{2e_i} - p_i^{2(e_i-1)}$ such elements of maximal order in $(\mathbb{Z}/p_i^{e_i}\mathbb{Z})^2$. We repeat this for each prime power factor of $N$ and by the Chinese Remainder Theorem we may take their product to get
  \begin{equation}\label{eq:num-points}
    N \phi(N) \prod_{p \mid N}\left(1 + \frac{1}{p}\right)
  \end{equation}
  total points of order $N.$ If $N = 2,$ then $P = -P$ for any point $P$ and so we may substitute $N = 2$ into \eqref{eq:num-points} to get a scaling factor of $3.$ Otherwise, if $N > 2,$ then we must divide the result of \eqref{eq:num-points} by $2$ due to identifying $P$ with $-P.$

  To prove (i), we use the fact that any cyclic subgroup $\left\langle P \right\rangle \subseteq E[N]$ of order $N$ has $\phi(N)$ elements with order $N.$ We may therefore partition the elements $P$ of order $N$ into equivalence classes where $P \sim P'$ if there exists an integer $k$ such that $P' = kP.$ From (ii), we have that there are
  \[
    N\phi(N)\prod_{p \mid N}\left( 1 + \frac{1}{p} \right)
  \]
  elements of order $N$ in $E[N].$ Dividing by $\phi(N)$ the size of each partition, we see there are
  \[
     N \prod_{p \mid N}\left( 1 + \frac{1}{p} \right)
  \]
  such equivalence classes, and therefore subgroups, of order $N.$

  To prove (iii), we first work with a factor $(\mathbb{Z}/p_i^{e_i}\mathbb{Z})^2$ of \eqref{eq:torsion-factors}. The result of (ii) gives us the number of ways to choose a first basis point $P$. The restrictions on choosing a second basis point $Q$ is that it must be of order $p_i^{e_i}$ and it cannot be congruent to a scalar multiple of $P$ modulo $p_i.$ The congruence requirement arises from the fact that $\mathbb{Z}/p_i^{e_i} \mathbb{Z}$ is local at $p_i.$ From (ii), we know there are $\phi(p_i^{e_i})$ elements of order $p_i^{e_i}$ in $\left\langle P \right\rangle$ and that each of those elements is congruent to $p_i^{e_i-1}$ distinct elements of order $p_i^{e_i}.$ This gives a total of  $\phi(p_i^{e_i})(p_i^{e_i} + p_i^{e_i-1}) - \phi(p_i^{e_i})p_i^{e_i-1} = p_i^{e_i}\phi(p_i^{e_i})$ choices for $Q.$ Taking the product over all $p_i$ by the Chinese Remainder Theorem gives
  \begin{equation}\label{eq:num-bases}
    N^2\phi(N)^2\prod_{p \mid N}\left( 1 + \frac{1}{p} \right)
  \end{equation}
  choices of bases. If $N = 2,$ then $P = -P$ and $Q = -Q$ for any choice of $P, Q,$ and so substituting $N = 2$ into \eqref{eq:num-bases} gives a scaling factor of $6.$ If $N > 2,$ we divide \eqref{eq:num-bases} by $2$ due to identifying the pair $(P, Q)$ with $(-P, -Q).$
\end{proof}

The graphs of \Cref{defn:level-struct-volcanoes} consist of components which are each a covering graph of $G_{\ell}.$ The crater $C_{\ell}$ is the defining characteristic of $G_{\ell},$ so the following work will only discuss craters unless otherwise stated.

\begin{defn}
  Let $G_{\ell}$ be an isogeny volcano and $G_{\ell,0}(N),$ $G_{\ell,1}(N),$ and $G_{\ell}(N)$ the graphs obtained as described in \Cref{defn:level-struct-volcanoes}. Then
  \begin{enumerate}[label=(\roman*)]
  \item $C_{\ell,0}(N)$ denotes the subgraph of $G_{\ell,0}(N)$ whose components are obtained by adding a $\Gamma_0(N)$-level structure to $C_{\ell};$
  \item $C_{\ell,1}(N)$ denotes the subgraph of $G_{\ell,1}(N)$ whose components are obtained by adding a $\Gamma_1(N)$-level structure to $C_{\ell};$ and,
  \item $C_{\ell}(N)$ denotes the subgraph of $G_{\ell}(N)$ whose are obtained by adding a $\Gamma(N)$-level structure to $C_{\ell}.$
  \end{enumerate}
\end{defn}
For lack of a better term, we will refer to the graphs $C_{\ell,0}(N),$ $C_{\ell,1}(N)$ and $C_{\ell}(N)$ as the \emph{craters} of $G_{\ell,0}(N),$ $G_{\ell,1}(N)$ and $G_{\ell}(N)$ respectively. The reader should note that since the graphs $G_{\ell,1}(N)$ and $G_{\ell}(N)$ are directed, their craters do not look like craters of standard isogeny volcanoes.

Let $\mathfrak{l}$ be a prime $\mathcal{O}$-ideal above $\ell.$ If $\ell$ is an inert prime, then $\mathfrak{l}$ does not induce a degree $\ell$ isogeny and $C_{\ell}$ is totally disconnected. Adding a level structure does not change these craters. As such, for the remainder of this paper we will only concern ourselves with primes $\ell$ which either split or ramify. From~\cite{Kohel}, we know the size of $C_{\ell}$ is equal to the order $n$ of $[\mathfrak{l}]$ in the class group $Cl(\mathcal{O}).$ In particular, we have $\mathfrak{l}^n = \lambda\mathcal{O}$ and $\overline{\mathfrak{l}}^n = \overline{\lambda}\mathcal{O}$ for some $\lambda, \overline{\lambda} \in \mathcal{O}^{\times}$ where $\mathfrak{l} = \overline{\mathfrak{l}}$ if $\ell$ is ramified.
\subsection{Crater graphs $C_{\ell,0}(N)$} We will begin by looking at the components with a $\Gamma_0(N)$-level structure. Let $E$ be an elliptic curve on a crater $C_{\ell}.$
\begin{defn}\label{defn:principal-vertex-0}

  Let $G$ be a $\Gamma_0(N)$-level structure of $E.$ Then $v = (E,G)$ is a vertex on some component $\mathcal{C}_{\ell,0}(N) \subseteq C_{\ell,0}(N).$ We say a vertex $v' \in \mathcal{C}_{\ell,0}(N)$ is a \emph{$\Gamma_0$-principal vertex of $v$} if $v'$ is of the form $(E, G')$ for some $G' \subseteq E[N]$ of order $N.$
\end{defn}
A crater $C_{\ell}$ can be described by the behaviour of $\ell$ in its class group. If $\ell$ splits as a product of non-principal ideals, the crater is a cycle of length equal to the order of a prime $\mathfrak{l}$ above $\ell$ in its class group. If $\ell$ splits as a product of principal ideals, the crater consists of a single vertex with two loops. If $\ell$ ramifies as a non-principal ideal, then the crater consists of an edge, and if $\ell$ ramifies as a principal ideal, then the crater is a single vertex with a loop. The following lemma and theorem are independent of the behaviour of $\ell$ in its class group.
\begin{lem}\label{lem:principal-vertex-0}
Let $v = (E,G)$ be a vertex on $\mathcal{C}_{\ell,0}(N) \subseteq C_{\ell,0}(N).$ If $v' = (E, G')$ is a $\Gamma_0$-principal vertex of $v,$ then there exists an endomorphism $\alpha \in \Endo(E)$ such that $\alpha(G) = G'$. Furthermore, there exists some positive integer $d$ such that $\mathfrak{l}^d = \alpha\mathcal{O}.$
\end{lem}
\begin{proof}
Any edge of $\mathcal{C}_{\ell,0}(N)$ represents an isogeny induced by $\mathfrak{l}$ which has finite order $n$ in $Cl(\mathcal{O}).$ Since $v, v'$ are on the same component, the path connecting them is the isogeny $\varphi_{\mathfrak{l}^d}$ induced by $\mathfrak{l}^d$ for some $d.$ This isogeny must be induced by a principal ideal $\alpha\mathcal{O}$ since $\varphi_{\mathfrak{l}^d}: E \to E,$ and so $\mathfrak{l}^d = \alpha\mathcal{O}$ where $n \mid d.$
\end{proof}
It is now clear why we use the term \emph{principal vertex} in \Cref{defn:principal-vertex-0}: paths from a vertex $v$ to a principal vertex arise from principal ideals. \Cref{lem:principal-vertex-0} tells us the size of a component $\mathcal{C}_{\ell,0}(N) \subseteq C_{\ell,0}(N)$ is equal to the order $n$ of $[\mathfrak{l}]$ in the class group scaled by the number of principal vertices for some fixed vertex $v.$

Given $\mathfrak{a}$ an $\mathcal{O}$-ideal, we will let $\phi_{\mathcal{O}}(\mathfrak{a})$ denote the size of the residue ring $(\mathcal{O}/\mathfrak{a})^{\times}.$ This can be thought of as a generalised version of Euler's totient function. A standard result~\cite[Thm.~1.19]{narkiewicz_ant} gives
\begin{equation}\label{eq:general-euler-phi}
  \phi_{\mathcal{O}}(\mathfrak{a}) = \mathcal{N}(\mathfrak{a}) \prod_{\mathfrak{p} \mid \mathfrak{a}}\left( 1 - \frac{1}{\mathcal{N}(\mathfrak{p})} \right)
\end{equation}
where $\mathcal{N}(\mathfrak{a})$ denotes the norm of $\mathfrak{a}.$ We will let $\phi_{\mathcal{O}}$ denote the totient function on $\mathcal{O}$-ideals and $\phi$ denote the usual Euler totient function. We remark that $\phi(N) \neq \phi_{\mathcal{O}}(N \mathcal{O})$ as the right-hand side depends on how the (rational) prime factors of $N$ behave in $\mathcal{O}$ while the left-hand side only depends on the (rational) prime factorisation of $N.$

\begin{thm}\label{thm:crater-structure-0}
Let $v = (E, G)$ be a vertex on a component $\mathcal{C}_{\ell,0}(N) \subseteq C_{\ell,0}(N)$ where $G = \left\langle P \right\rangle$ and let $\mathfrak{a} = (P, N)$ be the $\mathcal{O}$-ideal generated by $P$ and $N.$ Recall that $\lambda$ was defined by $\mathfrak{l}^n = \lambda \mathcal{O}.$ Then $\#\mathcal{C}_{\ell,0}(N) = nm_{\mathfrak{a}}$ where $m_{\mathfrak{a}}$ is the order of the natural projection of $\lambda$ in $(\mathcal{O}/NO)^{\times}/\Stab(G).$ Further, $C_{\ell,0}(N)$ contains $\frac{\phi_{\mathcal{O}}(\mathfrak{a}^{-1}N \mathcal{O})}{\phi(N)m_{\mathfrak{a}}}$ components isomorphic to $\mathcal{C}_{\ell,0}(N).$
\end{thm}
\begin{proof}
  We first note that any component $\mathcal{C}_{\ell,0}(N)$ of $C_{\ell,0}(N)$ will have a crater whose size is divisible by the order $n$ of $\mathfrak{l}$ in $Cl(\mathcal{O}).$ Using this fact combined with \Cref{lem:principal-vertex-0}, it suffices to know how many principal vertices exist for some fixed vertex $v.$ Since paths between principal vertices correspond to principal ideals of the form $\lambda^k \mathcal{O}$ for some $k \in \mathbb{Z},$ we need to find the smallest $k$ such that $\lambda^kG = G$ to determine the component sizes. \Cref{prop:stab-gamma-0} tells us the stabilisers of $G,$ and so the number of principal vertices of $v$ is simply the order of the natural projection of $\lambda$ in $(\mathcal{O}/N\mathcal{O})^{\times}/\Stab(G).$

  To prove the second claim, observe that $\mathfrak{a}$ is the smallest ideal containing $P$ and $N,$ so by the isomorphism in \Cref{lem:O/N isomorphism}, the image of $P$ under the natural projection is in $(\mathcal{O}/\mathfrak{a}^{-1}N \mathcal{O})^{\times}$. By \eqref{eq:general-euler-phi}, the ring $(\mathcal{O}/\mathfrak{a}^{-1}N \mathcal{O})$ contains $\phi_{\mathcal{O}}(\mathfrak{a}^{-1}N \mathcal{O})$ such elements, one of which is the image of $P.$ For a cyclic group $G$ of order $N,$ there are $\phi(N)$ elements of order $N,$ and so we may partition these $\phi_{\mathcal{O}}(\mathfrak{a}^{-1}N \mathcal{O})$ elements whereby two points are equivalent if and only if they generate the same cyclic group. The result follows since any component contains $m_{\mathfrak{a}}$ such cyclic subgroups belonging to $E[N].$
\end{proof}
A consequence of \Cref{lem:principal-vertex-0} and \Cref{thm:crater-structure-0} is that if $\ell$ decomposes as a product of principal ideals, then the crater may consist of a cycle rather than loops when adding a $\Gamma_{0}(N)$ level structure.

By \Cref{thm:crater-structure-0}, we can determine the components (and thus all) of $G_{\ell,0}(N)$ by considering all ideals (including all of $\mathcal{O}$) $\mathfrak{a}$ dividing $N \mathcal{O}$ along with the condition that $N \mid \mathcal{N}(\mathfrak{a}^{-1}N),$ and looking at the order $m_{\mathfrak{a}}$ of the projection of $\lambda$ in $(\mathcal{O}/\mathfrak{a}^{-1}N \mathcal{O})^{\times}$ divided by the index $[(\mathcal{O}/\mathfrak{a}^{-1}N \mathcal{O})^{\times} : (\mathbb{Z}/N \mathbb{Z})^{\times}].$ The divisibility condition on $\mathcal{N}(\mathfrak{a}^{-1}N)$ arises from the fact that $P$ must generate a cyclic group of order $N.$ Equivalently, we can determine $G_{\ell,0}(N)$ by considering each $\Gamma_0(N)$-level structure and applying \Cref{thm:crater-structure-0} directly.

If we only consider the case where $N$ is a rational prime, the situation is simpler and we can describe component sizes of $C_{\ell,0}(N).$
\begin{cor}\label{cor:crater-structure-0}
  Let $N$ be a prime different from $\ell$ and $p$ and let $m$ denote the order of the natural projection of $\lambda$ in $(\mathcal{O}/N\mathcal{O})^{\times}/(\mathbb{Z}/N \mathbb{Z})^{\times}.$ Then $C_{\ell,0}(N)$ consists of one of the following:
  \begin{enumerate}[label=(\roman*)]
  \item If $N$ is inert, then $C_{\ell,0}(N)$ contains $\frac{N+1}{m}$ components each of size $nm.$
  \item If $N\mathcal{O} = \mathfrak{N}_1 \mathfrak{N}_2,$ then $C_{\ell,0}(N)$ contains $\frac{N-1}{m}$ components each of size $nm$ and two components of size $n.$
  \item If $N\mathcal{O} = \mathfrak{N}^2,$ then $C_{\ell,0}(N)$ contains $\frac{N}{m}$ components each of size $nm$ and one component of size $n.$
  \end{enumerate}
\end{cor}

\begin{proof}
  These are all special cases of \Cref{thm:crater-structure-0}.

  To prove (i), observe that $\mathcal{O}/N \mathcal{O}$ is isomorphic to a finite field of order $N^2,$ and so for any $P$ of order $N,$ we have $\mathfrak{a} = (P, N) = \mathcal{O}.$ Then $\#(\mathcal{O}/N)^{\times} = N^2-1$ and $\Stab(G) = (\mathbb{Z}/N \mathbb{Z})^{\times}$ for any $G.$ By \Cref{thm:crater-structure-0}, $\mathcal{C}_{\ell,0}(N),$ contains $\frac{N^2-1}{m(N-1)}$ components each of size $nm.$ To prove (ii), we first observe that $\lvert (\mathcal{O}/N \mathcal{O})^{\times} \rvert = (N-1)^2$ and that $\mathfrak{a} \in \left\{ \mathfrak{N}_1, \mathfrak{N}_2, \mathcal{O} \right\}.$ If $\mathfrak{a} = \mathfrak{N}_1$ or $\mathfrak{N}_2,$ then by \Cref{thm:crater-structure-0} we can determine the corresponding components by computing the order of the projection of $\lambda$ in $(\mathcal{O}/\mathfrak{a}^{-1}N \mathcal{O})^{\times}/(\mathbb{Z}/N \mathbb{Z})^{\times},$ but this is the trivial group, and so we obtain two components each of size $n.$ If $\mathfrak{a} = \mathcal{O},$ then $G$ has a generator corresponding to an element of $(\mathcal{O}/N \mathcal{O})^{\times},$ and so there are $\frac{(N-1)^2}{m(N-1)}$ components each of size $nm.$ To prove (iii), we proceed as in the split case but set $\mathfrak{N}_{1} = \mathfrak{N}_{2} = \mathfrak{N}$ and observe that $\lvert (\mathcal{O} / N\mathcal{O})^{\times} \rvert = N^{2} - N.$ If $\mathfrak{a} = \mathfrak{N},$ then $(\mathcal{O}/\mathfrak{a}^{-1}N\mathcal{O})^{\times} = (\mathcal{O} / \mathfrak{N})^{\times} \simeq (\mathbb{Z} / N\mathbb{Z})^{\times}$ and so we obtain one component of size $n.$ If $\mathfrak{a} = \mathcal{O},$ then we obtain $\frac{N^{2}-N}{m(N-1)} = \frac{N}{m}$ components of size $nm.$
\end{proof}

\subsection{Crater graphs $C_{\ell,1}(N)$}
\begin{defn}\label{defn:principal-vertex-1}
  Let $P$ be a $\Gamma_1(N)$-level structure of $E.$ Then $v = (E, P)$ is a vertex on some component $\mathcal{C}_{\ell,1}(N) \subseteq C_{\ell,1}(N).$ We say a vertex $v' \in \mathcal{C}_{\ell,1}(N)$ is a \emph{$\Gamma_1$-principal vertex of $v$} if it is of the form $(E, P')$ for some $P' \in E[N]$ of order $N.$
\end{defn}
\begin{lem}\label{lem:principal-vertex-1}
Let $v = (E,P)$ be a vertex on $\mathcal{C}_{\ell,1}(N) \subseteq C_{\ell,1}(N).$ If $v' = (E, P')$ is a $\Gamma_1$-principal vertex of $v,$ then there exists an endomorphism $\alpha \in \Endo(E)$ such that $\alpha(P) = \pm P'$. Furthermore, there exists some ideal $\mathfrak{b}$ in the group $\left\langle \mathfrak{l}, \overline{\mathfrak{l}} \right\rangle$ such that $\mathfrak{b} = \alpha\mathcal{O}.$
\end{lem}
\begin{proof}
  Any edge of $\mathcal{C}_{\ell,1}(N)$ arises from an isogeny induced by $\mathfrak{l}$ or $\overline{\mathfrak{l}}.$ Since $C_{\ell,1}(N)$ is a directed graph, then an isogeny induced by $\overline{\mathfrak{l}}$ does not necessarily represent a reverse edge to that arising from $\mathfrak{l}$ as in the case of $G_{\ell}$ or $G_{\ell,0}(N).$ Any isogeny mapping $v$ to $v'$ must be an endomorphism since it maps a point $P$ on $E$ to a point $P'$ which also lies on $E$ and thus proves the first claim. $\alpha$ be such an endomorphism.

  Any path between principal vertices is the result of an isogeny induced by some ideal $\mathfrak{b}$ belonging to the group $\left\langle \mathfrak{l}, \overline{\mathfrak{l}}  \right\rangle$. Since $v$ and $v'$ are on the same component, they are connected by a path arising from $\mathfrak{b}$ where $\mathfrak{b} = \alpha \mathcal{O}.$
\end{proof}

\begin{thm}\label{thm:crater-structure-1}
  Let $v = (E, P)$ be a vertex on a component $\mathcal{C}_{\ell,1}(N) \subseteq C_{\ell,1}(N)$ where $P \in E[N]$ is of order $N,$ and let $\mathfrak{a} = (P, N)$ be the $\mathcal{O}$-ideal generated by $P$ and $N.$ Then $\#\mathcal{C}_{\ell,1}(N) = nm_{\mathfrak{a}}$ where $m_{\mathfrak{a}}$ is the order of the group generated by the natural projections of $\lambda$, $\overline{\lambda}$ in $(\mathcal{O}/NO)^{\times}/\Stab(P).$ Further,  $C_{\ell,1}(N)$ contains $\frac{\phi_{\mathcal{O}}(\mathfrak{a}^{-1}N \mathcal{O})}{2m_{\mathfrak{a}}}$ components isomorphic to $\mathcal{C}_{\ell,1}(N).$
\end{thm}
\begin{proof}
  The proof follows that of \Cref{thm:crater-structure-0}, but with making the necessary change in stabiliser subgroup and utilising \Cref{lem:principal-vertex-1}. As observed in \Cref{lem:principal-vertex-1}, we must look at the group generated by the natural projections of $\lambda$, $\overline{\lambda}$ since $C_{\ell,1}(N)$ is a directed graph. The difference in the second claim is that we instead partition $\frac{\phi_{\mathcal{O}}(\mathfrak{a}^{-1}N \mathcal{O})}{2}$ elements where the denominator of $2$ comes from the fact that elements are equivalent if and only if they differ by a factor of $-1.$
\end{proof}
As in \Cref{cor:crater-structure-0}, if we restrict ourselves to the case where $N$ is prime, then we can describe the component sizes of  $C_{\ell,1}(N).$
\begin{cor}\label{cor:crater-structure-1}
  Let $N$ be a prime different from $\ell$ and $p,$ and let $m$ denote the order of the group generated by the natural projections of $\lambda$, $\overline{\lambda}$ in $(\mathcal{O}/N \mathcal{O})^{\times}/\left\{ \pm 1 \right\}.$ Then the graph $G_{\ell,1}(N)$ consists of one of the following:
  \begin{enumerate}[label=(\roman*)]
  \item If $N$ is inert, then there are $\frac{N^2-1}{2m}$ components each of size $nm.$
  \item If $N\mathcal{O} = \mathfrak{N}_1 \mathfrak{N}_2,$ then there are $\frac{(N-1)^2}{2m}$ components each of size $nm$, $\frac{(N-1)}{2m_1}$ components each of size $nm_1$ and $\frac{N-1}{2m_2}$ components each of size $nm_2$ where $m_1$, $m_2$ are the orders of the groups generated by the natural projections of $\lambda$, $\overline{\lambda}$ in $(\mathcal{O}/\mathfrak{N}_1)^{\times}$ and $(\mathcal{O}/\mathfrak{N}_2)^{\times}$ respectively.
  \item If $N\mathcal{O} = \mathfrak{N}^2,$ then there are $\frac{N^2-N}{2m}$ components each of size $nm$ and $(N-1)/2m$ components each of size $nm_1$ where $m_1$ is the order of the group generated by the natural projections of $\lambda$, $\overline{\lambda}$ in $(\mathcal{O}/\mathfrak{N})^{\times}/\left\{ \pm 1 \right\}.$
  \end{enumerate}
\end{cor}

\begin{proof}
  The proof of (i) is similar to that of \Cref{cor:crater-structure-0}~(i), and so we will focus on (ii) and (iii).

  As in \Cref{cor:crater-structure-0}, we fix a vertex $v = (E, P)$ where $E$ is on $C_{\ell}$ and $P$ is an element of $\mathcal{O}/N\mathcal{O}$ corresponding to a $\Gamma_1(N)$-level structure of $E.$ In the case where $N\mathcal{O} = \mathfrak{N}_1 \mathfrak{N}_2,$ if $P \in (\mathcal{O}/N\mathcal{O})^{\times},$ then by \Cref{prop:stab-gamma-1} we have $\Stab(P) = \left\{ \alpha \in \mathcal{O}/N \mathcal{O} \mid \alpha \equiv \pm 1 \mod N \right\}.$ The component containing $v$ will then be of size $nm.$ As $\phi_{\mathcal{O}}(N \mathcal{O}) = (N-1)^2,$ and we identify $P$ with $-P,$ we will have $\frac{(N-1)^2}{2m}$ components of size $nm.$ If instead $P \in \mathfrak{N}_1$ (resp. $\mathfrak{N}_2$), then by \Cref{prop:stab-gamma-1}, we instead have $\Stab(P) = \left\{ \alpha \in \mathcal{O}/N \mathcal{O} \mid \alpha \equiv \pm 1 \mod{\mathfrak{N}_2} \right\}$ (resp. $\left\{ \alpha \in \mathcal{O}/N \mathcal{O} \mid \alpha \equiv \pm 1 \mod{\mathfrak{N}_1} \right\}$). By the CRT, $(\mathcal{O}/N\mathcal{O})^{\times} \simeq (\mathcal{O}/\mathfrak{N}_1)^{\times} \times \mathcal{O}/(\mathfrak{N}_2)^{\times}$ and under this isomorphism $P$ may be written as $(0 \mod{\mathfrak{N}_{1}}, P \mod{\mathfrak{N}_2})$ (resp. $(P \mod{\mathfrak{N}_1}, 0 \mod{\mathfrak{N}_{2}})$). It follows that the component containing $v$ will be of size $nm_2$ (resp. $nm_1$). Since $\# (\mathcal{O}/\mathfrak{N}_1)^{\times}/\left\{ \pm 1 \right\} = \#(\mathcal{O}/\mathfrak{N}_2)^{\times}/\left\{ \pm 1 \right\} = \frac{N-1}{2},$ it follows that there are $\frac{N-1}{2m_2}$ (resp. $\frac{N-1}{2m_1}$) components of size $nm_2$ (resp. $nm_1$).

  The proof of (iii) is the same, but with the change that $\phi_{\mathcal{O}}(N \mathcal{O}) = N^2-N.$
\end{proof}

\subsection{Crater graphs $C_{\ell}(N)$}
\begin{defn}\label{defn:principal-vertex-full}
  Let $(P, Q)$ be $\Gamma(N)$-level structure of $E.$ Then the vertex $v = (E, P, Q)$ is on some component $\mathcal{C}_{\ell}(N) \subseteq C_{\ell}(N).$ We say a vertex $v' \in \mathcal{C}_{\ell}(N)$ is a \emph{$\Gamma$-principal vertex of $v$} if it is of the form $(E, P', Q')$ for some $P', Q' \in E[N]$ which are a basis for $E[N].$
\end{defn}

\begin{lem}\label{lem:principal-vertex-full}
Let $v = (E,P,Q)$ be a vertex on $\mathcal{C}_{\ell}(N) \subseteq C_{\ell}(N).$ If $v' = (E, P',Q')$ is a $\Gamma$-principal vertex of $v,$ then there exists an endomorphism $\alpha \in \Endo(E)$ such that $\alpha(P) = \pm P'$ and $\alpha(Q) = \pm Q'$. Furthermore, there exists some ideal $\mathfrak{b}$ in the group $\left\langle \mathfrak{l}, \overline{\mathfrak{l}} \right\rangle$ such that $\mathfrak{b} = \alpha\mathcal{O}.$
\end{lem}
\begin{proof}
  The proof is similar to that of \Cref{lem:principal-vertex-1}.
\end{proof}

As we will see shortly, the full level $N$ structure case turns out to be the simplest of the three cases to describe.

\begin{thm}\label{thm:crater-structure-full}
  Let $v = (E, P, Q)$ be a vertex on a component $\mathcal{C}_{\ell}(N) \subseteq C_{\ell}(N)$ where $P, Q$ are a $\mathbb{Z}$-basis for $E[N].$ Then $\# \mathcal{C}_{\ell}(N) = n m $ where $m$ is the order of the group generated by the natural projections of $\lambda$, $\overline{\lambda}$ in $(\mathcal{O}/N \mathcal{O})^{\times}/{\pm 1}.$ Further, $C_{\ell}(N)$ contains
  \[
    \frac{N^2\phi(N)^2}{2m}\prod_{p \mid N}\left( 1 + \frac{1}{p} \right)
  \]
  components isomorphic to $\mathcal{C}_{\ell}(N).$
\end{thm}
\begin{proof}
  We first fix a vertex $v = (E, P, Q)$ where $E$ is on $C_{\ell}$ and $P, Q$ are a $\mathbb{Z}$-basis for $E[N].$ From \Cref{prop:stab-full-level}, we have that $\Stab((P,Q)) = \left\{ \alpha \in \mathcal{O}/N \mathcal{O} \mid \alpha \equiv \pm 1 \mod N \mathcal{O} \right\}.$ The number of principal vertices of $v$ is given by the order $m$ of the group generated by the projections of $\lambda$, $\overline{\lambda}$ in $(\mathcal{O}/N\mathcal{O})^{\times}/\left\{ \pm 1 \right\}$ and so the component containing $v$ is of size $nm.$ \Cref{lem:volcano-growth}~(iii) gives us the number of vertices containing $E,$ and division by $m$ gives the desired result.
\end{proof}
We immediately see a difference between \Cref{thm:crater-structure-full} and the theorems describing $G_{\ell,0}(N)$ and $G_{\ell,1}(N)$ graphs: the structure is independent of the prime ideal factorisation of $N \mathcal{O}.$ However, given a vertex $v = (E, P, Q),$ the factorisation of $N \mathcal{O}$ can place restrictions on the possible principal vertices of $v.$ For example, if $N = \mathfrak{N}_1 \mathfrak{N}_2$ and $P \in \mathfrak{N}_1, Q \in (\mathcal{O}/N\mathcal{O})^{\times},$ then any principal vertex $v' = (E, P', Q')$ of $v$ must have the same structure (i.e. $P' \in \mathfrak{N}_1, Q' \in (\mathcal{O}/N\mathcal{O})^{\times}$).

\section{Class Field Theory}\label{sec:class-field-theory}
We now extend these results and use the language of class field theory to describe a group action on the various level structures.  We will begin by briefly discussing congruence subgroups. Much of the content we discuss here is covered in more detail in~\cite{cox}.

\subsection{Generalised Ideal Class Groups}
Let $K$ be an imaginary quadratic field and $\mathcal{O}$ an order in $K.$ We will let $\mathcal{I}_{\mathcal{O}}$ denote the group of proper fractional $\mathcal{O}$-ideals and $\mathcal{P}_{\mathcal{O}} \subseteq \mathcal{I}_{\mathcal{O}}$ the subgroup of fractional principal $\mathcal{O}$-ideals. In the imaginary quadratic setting\footnote{For general number fields, one needs to consider the real infinite primes dividing $\mathfrak{m}.$ Our focus is only on the imaginary quadratic setting where $K$ has no real infinite primes.}, a modulus $\mathfrak{m} \subseteq \mathcal{O}$ is simply an $\mathcal{O}$-ideal. For a modulus $\mathfrak{m},$ we let $\mathcal{I}_{\mathcal{O}}(\mathfrak{m})$ denote the group of proper fractional $\mathcal{O}$-ideals which are coprime to $\mathfrak{m},$ $\mathcal{P}_{\mathcal{O}}(\mathfrak{m}) \subseteq \mathcal{I}_{\mathcal{O}}(\mathfrak{m})$ denote the subgroup of principal ideals which are coprime to $\mathfrak{m},$ and $\mathcal{P}_{\mathcal{O},1}(\mathfrak{m}) \subseteq \mathcal{P}_{\mathcal{O}}$ denote the subgroup of principal ideals generated by $\alpha \in K^{\times}$ such that $\alpha \equiv 1 \mod \mathfrak{m}.$ A subgroup $H \subseteq \mathcal{I}_{\mathcal{O}}(\mathfrak{m})$ satisfying $\mathcal{P}_{\mathcal{O},1}(\mathfrak{m}) \subseteq H \subseteq \mathcal{I}_{\mathcal{O}}(\mathfrak{m})$ is called a \emph{congruence subgroup} for $\mathfrak{m}.$ Given a congruence subgroup $H,$ we may form the quotient group
\[
  Cl_{\mathcal{O}}(\mathfrak{m}) = \mathcal{I}_{\mathcal{O}}(\mathfrak{m}) / H
\]
which we call a \emph{generalised ideal class group} for $\mathfrak{m}.$

A trivial example of a congruence subgroup we are familiar with is obtained by setting $\mathfrak{m} = 1.$ In this setting, any principal ideal $(\alpha)$ satisfies $\alpha \equiv 1 \mod \mathfrak{m}$ and so $\mathcal{P}_{\mathcal{O},1} = \mathcal{P}_{\mathcal{O}}.$ Then the generalised ideal class group is the usual ideal class group $Cl_{\mathcal{O}} = \mathcal{I}_{\mathcal{O}}/\mathcal{P}_{\mathcal{O}}$ we are familiar with.

In this paper, we focus on the setting $\mathfrak{m} = N\mathcal{O}$ where $N$ is defined as in the previous sections. The two congruence subgroups we are interested in are $\mathcal{P}_{\mathcal{O},1}(N),$ and
\[
\mathcal{P}_{\mathcal{O},\mathbb{Z}}(N) = \left\{ (\alpha) \mid \alpha \equiv c \mod N\mathcal{O} \text{ for } c \in \mathbb{Z} \text{ with } (c, N) = 1 \right\}.
\]
The first of these congruence subgroups is generated by the elements which fix all of a $\Gamma_1(N)$ and $\Gamma(N)$-level structure, while the second fixes all of a $\Gamma_0(N)$-level structure.

\begin{prop}\label{prop:cong-subgroup-integers}
  Let $\mathcal{O}, \mathcal{O}'$ be orders such that $\mathcal{O}' \subseteq \mathcal{O}$ and $[\mathcal{O}:\mathcal{O}'] = f.$ Then
  \[
    Cl_{\mathcal{O}'} \simeq \mathcal{I}_{\mathcal{O}}(f)/\mathcal{P}_{\mathcal{O},\mathbb{Z}}(f).
  \]
\end{prop}
\begin{proof}
  As stated in~\cite[Thm.~2.3.1]{class-group-action-oriented}, the proof is the relative version of~\cite[Prop.~7.22]{cox}.
\end{proof}

We are now ready to discuss the group action on level structures. Given an ordinary elliptic curve $E$ with $\Endo(E) \simeq \mathcal{O}$ and a level structure $\gamma(N)$ on $E$ of type $\Gamma_0(N),$ $\Gamma_1(N),$ or $\Gamma(N),$ we may define a group action
\[
  \mathfrak{a} \cdot \gamma(N) = \varphi_{\mathfrak{a}}(\gamma(N))
\]
where $\mathfrak{a} \in \mathcal{I}_{\mathcal{O}}(N)$ and $\varphi_{\mathfrak{a}}$ is the isogeny induced by $\mathfrak{a}$ whose kernel is given by $E[\mathfrak{a}].$ Up to the equivalence relation defined in \Cref{def:vertex-equivalence}, if $\gamma(N)$ is a $\Gamma_0(N)$-level structure, then the common stabilisers among all $\Gamma_0(N)$-level structures is the set $\mathcal{P}_{\mathcal{O},\mathbb{Z}}(N)$ of principal ideals which act as scalar multiplication on $\gamma(N)$; if it is a $\Gamma_1(N)$-level structure, then common stabilisers is the set $\mathcal{P}_{\mathcal{O},1}(N)$ of principal ideals which act as multiplication by $\pm 1$ on $\gamma(N);$ and if it is a $\Gamma(N)$-level structure, its stabilisers is the set $\mathcal{P}_{\mathcal{O},1}(N)$. We can then form quotient groups
\begin{align}
  Cl_{N\mathcal{O}}(N) &= \mathcal{I}_{N\mathcal{O}}(N) / \mathcal{P}_{NO}(N) \text{ and } \label{eq:cls-grp-NO}\\
  Cl_{\mathcal{O},1}(N) &= \mathcal{I}_{\mathcal{O}}(N) / \mathcal{P}_{\mathcal{O},1}(N) \label{eq:ray-cls-grp}
\end{align}
where (\eqref{eq:cls-grp-NO}) acts on isomorphism classes of $\Gamma_0(N)$-level structures and (\eqref{eq:ray-cls-grp}) acts on isomorphism classes of $\Gamma_1(N)$ and $\Gamma(N)$-level structures. The first of these two generalised ideal class groups is simply the ideal class group for the order of index $N$ in $\mathcal{O},$ while the second is known as the \emph{ray class group} for modulus $\mathfrak{m} = N.$
\section{Examples}\label{sec:examples}
We will now look at several examples. We will add various level structures to craters. The examples covered are computed using Sagemath~\cite{sage}. The craters with level structure are constructed by explicitly computing $\ell$-isogenies and applying them to the various level structures we are interested in. The examples we look at will all consider the case where $\ell$ splits in $\mathcal{O}$ as $\ell\mathcal{O} = \mathfrak{l}\cdot\overline{\mathfrak{l}}.$ We will lift $\mathfrak{l}$ to the various groups discussed in \Cref{sec:class-field-theory} and compute its order to show consistency with the craters we construct.

\begin{eg}\label{eg:example-1}
  Let $p = 107,$ $\ell = 5,$ $N = 6,$ and $E/\mathbb{F}_p$ defined by $y^2 = x^3 + 43x + 86$ where $j(E) = 19.$ We have the discriminant of Frobenius $\Delta_{\pi} = -284 = 2^2(-71)$ and so $E$ has CM by an order $\mathcal{O}$  in $K = \mathbb{Q}(\sqrt{-71}).$ One can check that $\Endo(E) = \mathcal{O}_K = \mathbb{Z}[\Phi] = \mathbb{Z}\left[\frac{1+\sqrt{-71}}{2}\right].$ \Cref{fig:eg-ell-crater-j19} shows the crater of the component of the $\ell$-isogeny graph containing $E$ with the blue vertex highlighting the isomorphism class $E$ belongs to while
\Cref{fig:eg-Gamma0-crater-j19} shows the crater graph $C_{5,0}(6)$ obtained by adding a $\Gamma_0(5)$-level structure to the graph in \Cref{fig:eg-ell-crater-j19}.
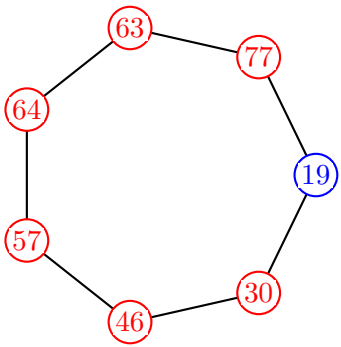
\begin{figure}
  \centering
  \begin{tikzpicture}
    \def\radius{2}
    \def\noderadius{8pt}
    \def\jinvariants{{19, 77, 63, 64, 57, 46, 30}}

    \foreach \i in {0,...,6} {
      \pgfmathsetmacro\a{\i * 360/7}
      \draw[thick, black] (\a:\radius) -- ({\a + 360/7}:\radius);
    }
    \foreach [count=\j from 0] \i in {0,...,6} {
      \pgfmathsetmacro\a{\i * 360/7}
      \pgfmathsetmacro{\jinvariant}{\jinvariants[\j]}
      \ifnum\i=0
        \draw[thick, blue, fill=white](\a:\radius) circle(\noderadius)
        node[anchor=center, fill=white, inner sep=1pt] {\jinvariant};
      \else
        \draw[thick, red, fill=white](\a:\radius) circle(\noderadius)
        node[anchor=center, fill=white, inner sep=1pt] {\jinvariant};
      \fi
    }
  \end{tikzpicture}
  \caption{Crater of 5-volcano containing $j(E) = 19$ over
    $\mathbb{F}_{107}.$}\label{fig:eg-ell-crater-j19}
\end{figure}

The factorisation of $\ell$ and $N$ into products of prime ideals is given by
  \begin{align*}
    \ell\mathcal{O} &= (5) = \mathfrak{l} \cdot \overline{\mathfrak{l}} = (5, \Phi+1)(5, \Phi+3) \\
    N\mathcal{O} &= (6) = \mathfrak{p}\cdot \overline{\mathfrak{p}} \cdot \mathfrak{q} \cdot \overline{\mathfrak{q}} = (2, \Phi)(2, \Phi + 1)(3,\Phi)(3,\Phi+2).
  \end{align*}
  The ideal class $[\mathfrak{l}]$ has order $7$ in  $Cl_K$ which is consistent with the crater size seen in \Cref{fig:eg-ell-crater-j19}.
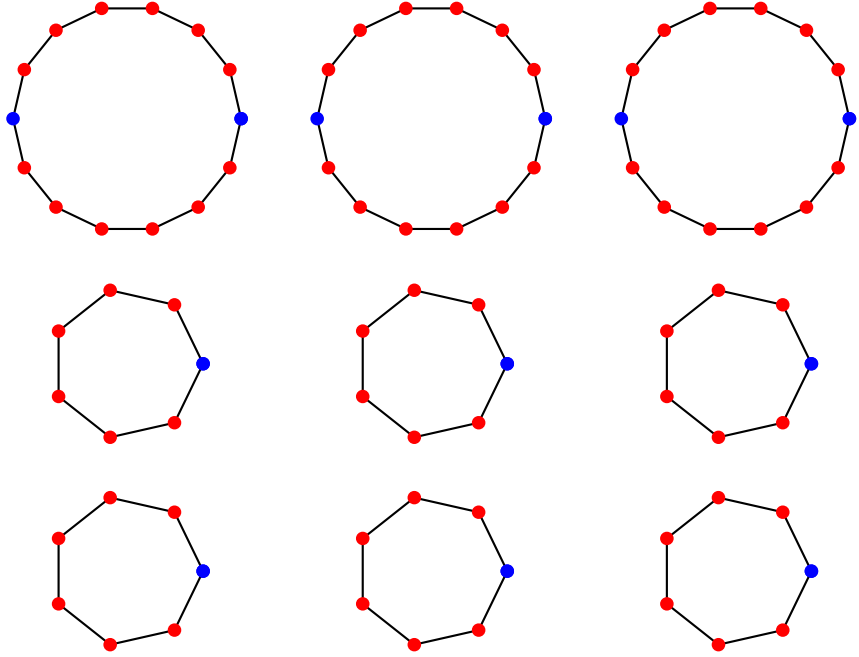
\begin{figure}
  \centering
  \begin{tikzpicture}
    \pgfmathsetmacro\A{360/14}
    \pgfmathsetmacro\a{360/7}
    \def\R{1.5}
    \def\r{1}
    \foreach \s in {0,1,2}{
      \foreach \v in {0,1,2,...,13} {
        \begin{scope}[xshift = \s*4cm]
          \pgfmathparse{int(mod(\v, 7))}
          \draw[thick, black] (\v*\A:\R) -- (\A+\v*\A:\R);
          \ifnum\pgfmathresult=0
            \draw[thick, blue, fill=blue] (\v*\A:\R) circle (\R/2/10);
          \else
            \draw[thick, red, fill=red] (\v*\A:\R) circle (\R/2/10);
          \fi
        \end{scope}
      }
      \begin{scope}[xshift = \s*4cm]
        \draw[thick, blue, fill=blue] (0:\R) circle (\R/2/10);
      \end{scope}
    }
    \foreach \s in {0,1,2}{
      \foreach \v in {0,1,2,...,6} {
        \begin{scope}[xshift = \s*4cm, yshift = -3.25cm]
          \pgfmathparse{int(mod(\v, 7))}
          \draw[thick, black] (\v*\a:\r) -- (\a+\v*\a:\r);
          \ifnum\pgfmathresult=0
            \draw[thick, blue, fill=blue] (\v*\a:\r) circle (\R/2/10);
          \else
            \draw[thick, red, fill=red] (\v*\a:\r) circle (\R/2/10);
          \fi\
        \end{scope}
        \begin{scope}[xshift = \s*4cm, yshift = -6cm]
          \pgfmathparse{int(mod(\v, 7))}
          \draw[thick, black] (\v*\a:\r) -- (\a+\v*\a:\r);
          \ifnum\pgfmathresult=0
            \draw[thick, blue, fill=blue] (\v*\a:\r) circle (\R/2/10);
          \else
            \draw[thick, red, fill=red] (\v*\a:\r) circle (\R/2/10);
          \fi
        \end{scope}
      }
      \begin{scope}[xshift = \s*4cm, yshift = -3.25cm]
        \draw[thick, blue, fill=blue] (0:\r) circle (\R/2/10);
      \end{scope}
      \begin{scope}[xshift = \s*4cm, yshift = -6cm]
        \draw[thick, blue, fill=blue] (0:\r) circle (\R/2/10);
      \end{scope}
    }
  \end{tikzpicture}
  \caption{Crater $C_{5,0}(6)$ corresponding to adding a $\Gamma_0(6)$-level structure to the $5$-isogeny graph containing $j(E) = 19.$ Blue vertices correspond to principal vertices of $E.$}\label{fig:eg-Gamma0-crater-j19}
\end{figure}For the ideal $\mathfrak{l},$ we have $\mathfrak{l}^7 = \lambda\mathcal{O} = (-4\Phi + 281).$ If $\left\langle P \right\rangle$ is a $\Gamma_0(N)$-level structure on $E,$ we have three cases to consider:
  \begin{enumerate}[label=(\roman*)]
  \item $P$ is identified with an element of $(\mathcal{O}/N\mathcal{O})^{\times};$
  \item $P$ is identified with an element of exactly one of $\mathfrak{p}, \overline{\mathfrak{p}}, \mathfrak{q},$ or $\overline{\mathfrak{q}}$ or;
  \item $P$ is identified with an element of exactly one of $\mathfrak{p}, \overline{\mathfrak{p}}$ and exactly one of $\mathfrak{q}, \overline{\mathfrak{q}}.$
  \end{enumerate}
  For the first case, we can determine the size of the component containing $\left\langle P \right\rangle$ by considering the order of the projection of $\lambda$ in $(\mathcal{O}/N\mathcal{O})^{\times}/(\mathbb{Z}/N \mathbb{Z})^{\times}.$ We have $\#({O}/N\mathcal{O})^{\times}/(\mathbb{Z}/N \mathbb{Z})^{\times} = 4/2 = 2$ and $\lambda = -4\Phi + 281 \equiv 2\Phi + 5$ modulo $6.$ This is clearly not congruent to an integer modulo $6,$ so $\lvert \lambda \rvert= 2$ and so by \Cref{thm:crater-structure-0} we have one component which is twice the size of our original volcano. For the second case, we need to consider the order of the projection of $\lambda$ in each of
  \begin{align*}
    &(\mathcal{O}/\mathfrak{p}^{-1}N\mathcal{O})^{\times}/(\mathbb{Z}/N \mathbb{Z})^{\times}, \\
    &(\mathcal{O}/\overline{\mathfrak{p}}^{-1}N\mathcal{O})^{\times}/(\mathbb{Z}/N \mathbb{Z})^{\times}, \\
    &(\mathcal{O}/\mathfrak{q}^{-1}N\mathcal{O})^{\times}/(\mathbb{Z}/N \mathbb{Z})^{\times}; \text{ and} \\
    &(\mathcal{O}/\overline{\mathfrak{q}}^{-1}N\mathcal{O})^{\times}/(\mathbb{Z}/N \mathbb{Z})^{\times}.\\
  \end{align*}
  The last two of these rings has trivial order and so each corresponds to a component isomorphic to our original volcano. The first two rings are of size $2,$ and so it suffices to check the projection of $\lambda$ in $(\mathcal{O}/\mathfrak{p}^{-1}N\mathcal{O})^{\times}$ and $(\mathcal{O}/\overline{\mathfrak{p}}^{-1}N\mathcal{O})^{\times}.$ We compute these projections to be $2\Phi - 1$ and $-\Phi - 1$ respectively which both square to the identity element in their respective rings. Therefore, we have an additional two components which are twice of our original volcano, giving a total of three components of this size. Finally, consider when $P \in \mathfrak{p} \mathfrak{q}.$ Then
  \[
    (\mathcal{O}/\mathfrak{p}^{-1}\mathfrak{q}^{-1}N\mathcal{O})^{\times} \simeq (\mathcal{O}/\overline{\mathfrak{p}}\overline{\mathfrak{q}})^{\times} \simeq (\mathbb{Z}/N \mathbb{Z})^{\times},
  \]
  and so $\lambda$ always acts as an integer on $\left\langle P \right\rangle.$ The result holds for all possibilities in the third case, and so we have four additional components isomorphic to the original volcano for a total of six copies of the original volcano.
\end{eg}

\begin{eg}
  Let $p = 47,$ $\ell = 5,$ $N = 3$ and $E/\mathbb{F}_p$ defined by $y^2 = x^3 + 14x + 5$ where $j(E) = 8.$ The discriminant of Frobenius is given by $\Delta_{\pi} = 2^2(-31).$ We compute the endomorphism ring of $E$ to be $\Endo(E) = \mathcal{O} = \mathbb{Z}[\pi].$
  \begin{figure}[!hbtp]
  \centering
    \begin{tikzpicture}
      \def\radius{2}
      \def\noderadius{8pt}
      \def\jinvariants{{8, 34, 29}}

      \foreach \i in {0,1,2} {
        \pgfmathsetmacro\a{\i * 360/3}
        \draw[thick, black] (\a:\radius) -- ({\a + 360/3}:\radius);
      }
      \foreach [count=\j from 0] \i in {0,1,2} {
        \pgfmathsetmacro\a{\i * 360/3}
        \pgfmathsetmacro{\jinvariant}{\jinvariants[\j]}
        \ifnum\i=0
          \draw[thick, blue, fill=white](\a:\radius) circle(\noderadius)
          node[anchor=center, fill=white, inner sep=1pt] {\jinvariant};
        \else
          \draw[thick, red, fill=white](\a:\radius) circle(\noderadius)
          node[anchor=center, fill=white, inner sep=1pt] {\jinvariant};
        \fi
      }
    \end{tikzpicture}
    \caption{Crater of 5-volcano containing $j(E) = 8$ over
    $\mathbb{F}_{47}.$}\label{fig:eg-ell-crater-j8}
  \end{figure}
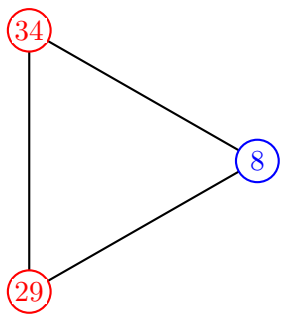
  Seen in \Cref{fig:eg-ell-crater-j8} is the crater containing the isomorphism class of $E$ and \Cref{fig:eg-Gamma0-crater-j8} shows the crater graph $C_{5,0}(3).$
  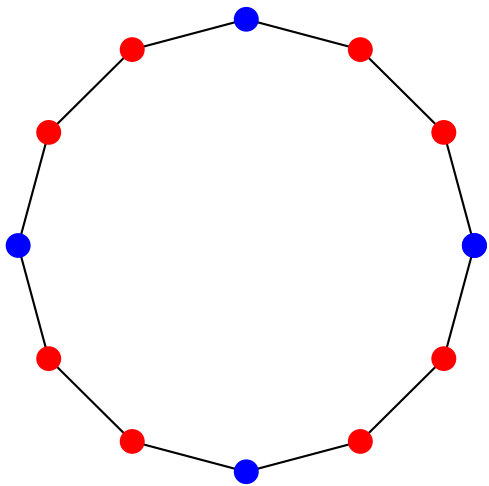
\begin{figure}[!hbtp]
    \centering
    \begin{tikzpicture}
      \pgfmathsetmacro\A{360/12}
      \pgfmathsetmacro\a{360/3}
      \def\R{3}
      \def\r{1}
        \foreach \v in {0,1,2,...,11} {
          \pgfmathparse{int(mod(\v, 3))}
          \draw[thick, black] (\v*\A:\R) -- (\A+\v*\A:\R);
          \ifnum\pgfmathresult=0
            \draw[thick, blue, fill=blue] (\v*\A:\R) circle (\R/2/10);
          \else
            \draw[thick, red, fill=red] (\v*\A:\R) circle (\R/2/10);
          \fi
        }
        \draw[thick, blue, fill=blue] (0:\R) circle (\R/2/10);
    \end{tikzpicture}
    \caption{Crater $C_{5,0}(3)$ corresponding to adding a $\Gamma_0(3)$-level structure to the $5$-isogeny graph containing $j(E) = 8.$ Blue vertices correspond to principal vertices of $E.$}\label{fig:eg-Gamma0-crater-j8}
  \end{figure}  A notable difference between the graph seen in \Cref{fig:eg-Gamma0-crater-j8} and the one seen in \Cref{fig:eg-Gamma0-crater-j19} of the previous example is the number of components. In this example,
  \[
    \ell\mathcal{O} = (5) = \mathfrak{l} \cdot \overline{\mathfrak{l}} = (5, \pi + 3)(5, \pi + 4)
  \]
  and $N$ is an inert prime, so any $\Gamma_0(N)$-level structure is generated by an element $P$ which can be identified with an element of $(\mathcal{O}/N\mathcal{O})^{\times}.$ Consequently, the crater graph $C_{5,0}(3)$ will consist of components all isomorphic to each other and by \Cref{prop:cong-subgroup-integers}, they will have size equal to the order of $[\mathfrak{l}]$ in $Cl(\mathcal{O}')$ where $\mathcal{O}' = \mathbb{Z} + 3 \mathcal{O} = \mathbb{Z} + 6 \mathcal{O}_K.$ We compute $h(O') = 12,$ and so we simply need to check whether which of $\mathfrak{l}^3,$ $\mathfrak{l}^6,$ and $\mathfrak{l}^{12}$ is principal when $\mathfrak{l}$ is viewed as an $\mathcal{O}'$-ideal. We compute
  \begin{align*}
    \mathfrak{l} &= (5, 3\pi + 2) \\
    \mathfrak{l}^3 &= (125, 3\pi + 52) \\
    \mathfrak{l}^6 &= (15625, 3\pi + 3802) \\
  \end{align*}
  none of which is principal, and so $[\mathfrak{l}]$ has order $12$ in $Cl(\mathcal{O}).$
\end{eg}

We now look at our final example which is a graph obtained by adding a $\Gamma_1(N)$-level structure.
\begin{eg}\label{eg:diff-crater-sizes}
  Let $p = 53,$ $\ell = 3,$ $N = 5,$ and $E/\mathbb{F}_p$ be defined by $y^2 = x^3 + 46x + 6$ where $j(E) = 8.$ The endomorphism ring of $E$ is given by $\Endo(E) = \mathcal{O} = \mathbb{Z} + 2 \mathbb{Z}[\Phi]$ where $\Phi = \left[ \frac{1+\sqrt{-11}}{2}\right].$

  \Cref{fig:eg-ell-crater-j8-2} shows the crater containing the isomorphism class of $E,$ and \Cref{fig:eg-Gamma1-crater-j8} shows the crater graph $C_{3,1}(5).$
  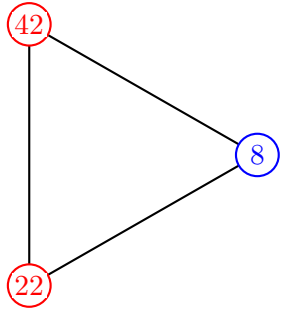
\begin{figure}[!hbtp]
    \centering
    \begin{tikzpicture}
      \def\radius{2}
      \def\noderadius{8pt}
      \def\jinvariants{{8, 42, 22}}

      \foreach \i in {0,1,2} {
        \pgfmathsetmacro\a{\i * 360/3}
        \draw[thick, black] (\a:\radius) -- ({\a + 360/3}:\radius);
      }
      \foreach [count=\j from 0] \i in {0,1,2} {
        \pgfmathsetmacro\a{\i * 360/3}
        \pgfmathsetmacro{\jinvariant}{\jinvariants[\j]}
        \ifnum\i=0
          \draw[thick, blue, fill=white](\a:\radius) circle(\noderadius)
          node[anchor=center, fill=white, inner sep=1pt] {\jinvariant};
        \else
          \draw[thick, red, fill=white](\a:\radius) circle(\noderadius)
          node[anchor=center, fill=white, inner sep=1pt] {\jinvariant};
        \fi
      }
    \end{tikzpicture}
    \caption{Crater of 3-volcano containing $j(E) = 8$ over
      $\mathbb{F}_{53}.$}\label{fig:eg-ell-crater-j8-2}
  \end{figure}

  \begin{figure}[!hbtp]
    \centering
    \begin{tikzpicture}[>=stealth]
      \pgfmathsetmacro\A{360/12}
      \pgfmathsetmacro\Amid{360/6}
      \pgfmathsetmacro\a{360/3}
      \def\R{1.5}
      \def\Rmid{1.25}
      \def\r{1}
      \foreach \s in {0,1}{
        \foreach \v in {0,1,2,...,11} {
          \begin{scope}[xshift = \s*4cm]
            \pgfmathparse{int(mod(\v, 3))}
            \draw[->, thick, black, shorten >=2pt] (\v*\A:\R) -- (\A+\v*\A:\R);
            \draw[dashed, ->, black, shorten >=2pt] (\v*\A:\R) -- (5*\A+\v*\A:\R);
            \ifnum\pgfmathresult=0
              \draw[thick, blue, fill=blue] (\v*\A:\R) circle (\R/2/10);
            \else
              \draw[thick, red, fill=red] (\v*\A:\R) circle (\R/2/10);
            \fi
          \end{scope}
        }
        \begin{scope}[xshift = \s*4cm]
          \draw[thick, blue, fill=blue] (0:\R) circle (\R/2/10);
        \end{scope}
      }
      \foreach \v in {0,1,2,...,5}{
        \begin{scope}[xshift = 0cm, yshift = -3.25cm]
          \pgfmathparse{int(mod(\v, 3))}
          \draw[->, thick, black, shorten >=2pt] (\v*\Amid:\Rmid) -- (\Amid+\v*\Amid:\Rmid);
          \draw[dashed, ->, black, shorten >=2pt] (\v*\Amid:\Rmid) -- (2*\Amid + \v*\Amid:\Rmid);
          \ifnum\pgfmathresult=0
            \draw[thick, blue, fill=blue] (\v*\Amid:\Rmid) circle (\Rmid/2/10);
          \else
            \draw[thick, red, fill=red] (\v*\Amid:\Rmid) circle (\Rmid/2/10);
          \fi
        \end{scope}
      }

      \foreach \v in {0,1,2,...,5}{
        \begin{scope}[xshift = 4cm, yshift = -3.25cm]
          \pgfmathparse{int(mod(\v, 3))}
          \draw[dashed, ->, black, shorten >=2pt] (\v*\Amid:\Rmid) -- (\Amid+\v*\Amid:\Rmid);
          \draw[->, thick, black, shorten >=2pt] (\v*\Amid:\Rmid) -- (2*\Amid + \v*\Amid:\Rmid);
          \ifnum\pgfmathresult=0
            \draw[thick, blue, fill=blue] (\v*\Amid:\Rmid) circle (\Rmid/2/10);
          \else
            \draw[thick, red, fill=red] (\v*\Amid:\Rmid) circle (\Rmid/2/10);
          \fi
        \end{scope}
      }
    \end{tikzpicture}
    \caption{Crater $C_{3,1}(5)$ corresponding to adding a $\Gamma_1(5)$-level structure to the $3$-isogeny graph containing $j(E) = 8.$ Blue vertices correspond to principal vertices of $E.$ Solid arrows represent isogenies induced by $\mathfrak{l},$ and dashed lines represent isogenies induced by $\overline{\mathfrak{l}}.$}\label{fig:eg-Gamma1-crater-j8}
  \end{figure}
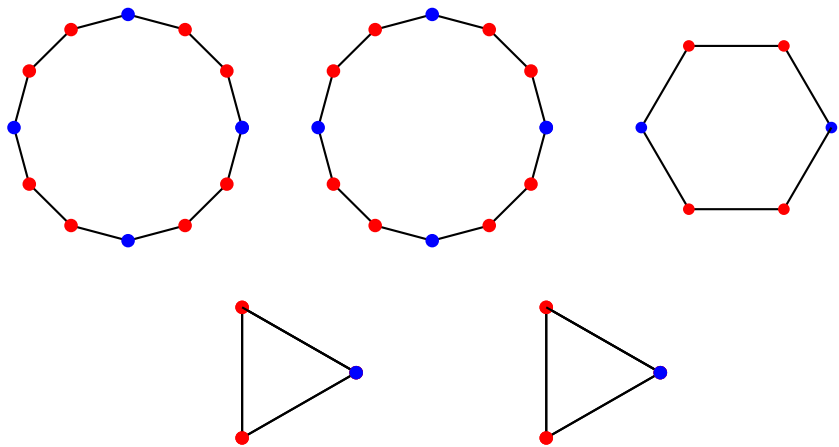

  We factor $\ell, N$ into prime $\mathcal{O}$-ideals as
  \begin{align*}
    \ell\mathcal{O} &= \mathfrak{l} \cdot \overline{\mathfrak{l}} = (3, 2\Phi + 1)(3, 2\Phi), \\
    N\mathcal{O} &= \mathfrak{N} \cdot \overline{\mathfrak{N}} = (5, 2\Phi + 1)(5, 2\Phi + 2).
  \end{align*}
  In $Cl(K),$ $[\mathfrak{l}]$ has order $3$ and $\mathfrak{l}^3 = \lambda\mathcal{O} = (-2\Phi + 5).$ For a $\Gamma_1(N)$-level structure $P,$ we have three cases to consider:
  \begin{enumerate}[label=(\roman*)]
  \item $P$ is identified with an element of $(\mathcal{O}/N\mathcal{O})^{\times};$
  \item $P$ is identified with an element of $\mathfrak{N}$ or;
  \item $P$ is identified with an element of $\overline{\mathfrak{N}}.$
  \end{enumerate}
  For the first case, we have $\#(\mathcal{O}/N\mathcal{O})^{\times}/2 = 8.$ We let $\tilde{\lambda}$ denote the equivalence class of $(\mathcal{O}/N\mathcal{O})^{\times}$ containing $\lambda$ and compute
  \begin{align*}
    \tilde{\lambda} &= -2\Phi + 5 \equiv -2\Phi \mod{N} \\
    \tilde{\lambda}^2 &= -16\Phi + 13 \equiv -\Phi + 3 \mod{N} \\
    \tilde{\lambda}^4 &= -160\Phi - 599 \equiv 1 \mod{N},
  \end{align*}
  and so $\tilde{\lambda}$ has order $4$ and we get two components of size $12.$ Performing the same computations with $\overline{\mathfrak{l}}$ we find that the projection of $\overline{\lambda}$ in $(\mathcal{O}/N\mathcal{O})^{\times}$ also has order $4.$ One may also show that $\overline{\mathfrak{l}} \equiv \mathfrak{l}^5 \mod N$ to construct the other half of the edges in \Cref{fig:eg-Gamma1-crater-j8}.

  The other two cases we need to consider are when $P$ corresponds to an element of either $\mathfrak{N}$ or $\overline{\mathfrak{N}}.$ If $P$ corresponds to an element of $\mathfrak{N},$ then we need to consider the order of $\tilde{\lambda}$ in $(\mathcal{O}/\overline{\mathfrak{N}})^{\times}.$ In the quotient ring $(\mathcal{O}/\overline{\mathfrak{N}})^{\times},$ $-2\Phi \equiv 3$ and so $\tilde{\lambda} \equiv 3$ which squares to $-1$ and stabilises $P.$ This gives us one component of size six. In the quotient ring $(\mathcal{O}/\mathfrak{N})^{\times},$ $-2\Phi \equiv 4$ and so $\tilde{\lambda} \equiv -1$ which gives two components of size $3.$ To obtain the other half of the edges, we perform the same computations with $\overline{\mathfrak{l}}$ and find that the projection of $\overline{\lambda}$ has order $3$ and $6$ in $(\mathcal{O}/\mathfrak{N})^{\times}$ and $(\mathcal{O}/\overline{\mathfrak{N}})^{\times}$ respectively and compute that $\mathfrak{l}^{2} \equiv \overline{\mathfrak{l}} \mod \mathfrak{N}$ and $\overline{\mathfrak{l}}^2 \equiv \mathfrak{l} \mod \overline{\mathfrak{N}}.$
\end{eg}

\bibliographystyle{plain}
\bibliography{refs}
\end{document}